\documentclass[12pt, reqno]{amsart}
\usepackage{amsmath, amsthm, amscd, amsfonts, amssymb, graphicx, color}
\usepackage{bm}
\usepackage[all,cmtip]{xy}

\newtheorem{theorem}{Theorem}[section]
\newtheorem{lemma}[theorem]{Lemma}
\newtheorem{proposition}[theorem]{Proposition}
\newtheorem{corollary}[theorem]{Corollary}
\theoremstyle{definition}
\newtheorem{definition}[theorem]{Definition}

\numberwithin{equation}{section}

\newcommand{\vp}{\varphi}

\newcommand{\clb}{\mathcal{B}}

\newcommand{\clf}{\mathcal{F}}

\newcommand{\clh}{\mathcal{H}}
\newcommand{\clj}{\mathcal{J}}
\newcommand{\clk}{\mathcal{K}}
\newcommand{\cll}{\mathcal{L}}
\newcommand{\clm}{\mathcal{M}}

\newcommand{\cls}{\mathcal{S}}

\newcommand{\clw}{\mathcal{W}}

\newcommand{\D}{\mathbb{D}}

\newcommand{\T}{\mathbb{T}}

\newcommand{\raro}{\rightarrow}

\newcommand{\C}{\mathbb{C}}

\newcommand{\hdcc}{{H}^2}

\newcommand{\frs}{\mathcal{T}_{[\vp; u,v]}}
\newcommand{\sfrs}{\mathcal{T}_{[\vp, u]}}

\begin{document}

\setcounter{page}{1}

\title[Invariant subspaces of perturbed backward shift]{Invariant subspaces of perturbed backward shift}

\author[Das]{Soma Das}
\address{Indian Statistical Institute, Statistics and Mathematics Unit, 8th Mile, Mysore Road, Bangalore, 560059, India}
\email{dsoma994@gmail.com, somadas\_ra@isibang.ac.in}

\author[Sarkar]{Jaydeb Sarkar}
\address{Indian Statistical Institute, Statistics and Mathematics Unit, 8th Mile, Mysore Road, Bangalore, 560059, India}

\email{jay@isibang.ac.in, jaydeb@gmail.com}




\begin{abstract}
We represent closed subspaces of the Hardy space that are invariant under finite-rank perturbations of the backward shift. We apply this to classify almost invariant subspaces of the backward shift and represent a more refined version of nearly invariant subspaces. Kernels of certain perturbed Toeplitz operators are examples of the newly introduced nearly invariant subspaces. 
\end{abstract}

\maketitle

\tableofcontents

\section{Introduction}\label{intro}

The famous ``invariant subspace problem'' calls for finding a nontrivial closed subspace that is invariant under a bounded linear operator on a Hilbert space. This problem is still open. It is even worse that we do not know how to solve this problem for a general rank-one perturbation of a diagonal operator \cite{Gallardo}. In the same vein, the lattice of invariant subspaces of a simple composition operator is unknown, and the issue is inextricably linked to the original invariant subspace problem \cite{Nord1}. Thus, the description of the lattices of invariant subspaces of common operators, or even the mere existence of invariant subspaces, is recognized as a complex problem. In particular, it is known to be a difficult challenge to find the lattice of invariant subspaces of finite-rank perturbations of well-known operators. In this paper, we will address this issue in relation to the well-known backward shift operator.

More specifically, in this paper, we bring together three problems concerning invariant subspaces of operators related to the backward shift operator $T_z^*$ acting on the Hardy space $H^2$. Recall the standard notation that $T_z$ is the multiplication operator by the coordinate function $z$; alternatively, the Toeplitz operator with the symbol $z$; and $H^2$ is the Hilbert space of all square summable analytic functions on the open unit disc $\D = \{z \in \C: |z| < 1\}$ (see \eqref{eqn: Hardy} for more details). Therefore
\[
T_z f = zf,
\]
for all $f \in H^2$. Strategically, we first represent closed invariant subspaces of finite-rank perturbations of $T_z^*$, and then we use that technique to classify almost invariant subspaces for $T_z^*$ and represent nearly invariant subspaces of $T_z^*$ (in a larger sense).

Before moving on, it is worth remarking that the results and approach of this paper are an enhancement of the theory that was presented by Hayashi \cite{Hayashi} and Hitt \cite{Hitt}, which was further fine-tuned by Sarason \cite{Sarason}. In fact, the idea we put forward for full-length parametrizations of invariant subspaces for all finite-rank perturbations of $T_z^*$ has been spurred by Sarason's insight regarding the invariance property of a particular rank-one perturbation of $T_z^*$ \cite[page 482]{Sarason} (also see \eqref{eqn: intro Tz M} for more details). We have subsequently laid out the road map to solve the other two problems.

To follow the progression of ideas, we begin with perturbations of $T_z^*$. Each pair of nonzero vectors $u$ and $v$ in $H^2$ yields the rank-one operator $u \otimes v$, where
\[
(u \otimes v)f = \langle f, v \rangle u,
\]
for all $f \in H^2$. The term rank-one is used to describe operators whose ranges are one-dimensional. Similarly, rank-$m$ operators are those whose ranges are $m$-dimensional. They can be represented as:
\[
\sum_{i=1}^{m} v_i \otimes u_i,
\]
where $\{u_i\}_{i=1}^m$ and $\{v_i\}_{i=1}^m$ are orthonormal sets of functions in $ H^2$. For the remainder of this section, we fix these orthonormal sets. We are interested in the invariant subspaces of a general finite-rank perturbation
\[
X = T_z^* + \sum_{i=1}^{m} v_i \otimes u_i.
\]
In Corollary \ref{cor: T_z per inv sub}, we present a complete description of invariant subspaces of $X$:

\begin{theorem}\label{thm: Intro 1}
Let $\clm$ be a closed subspace of $H^2$. If $\clm$ is invariant under $X$, then there exist a natural number $p' \leq p+1$ and an inner function $\Psi \in H^\infty_{\clb(\C^{p'}, \C^{p+1})}$ such that
\[
\clm = [F_0,1]_{1\times (p+1)} \clk_\Psi,
\]
where $\clk_\Psi = H^2_{\C^{p+1}}/\Psi H^2_{\C^{p'}}$ is the model space, $F_0 = [f_1, \ldots, f_p]$, and $\{f_i\}_{i=1}^p$ is an orthonormal basis for $\text{span}\{P_\clm u_1, \ldots ,P_\clm u_m\}$. Moreover
\[
\mathcal{K}_\Psi = \{(F,f_0)\in H^2_{\C^p} \oplus H^2 : F_0 F + f_0 \in \clm \textit{ and } \|F_0 F + f_0\|^2 = \|F\|^2 +\|f_0\|^2\}.
\]
\end{theorem}

In the given context, $H^2_{\C^p}$ represents the $\C^p$-valued Hardy space, which may also be seen as $p$-copies of $H^2$. Given Hilbert spaces $\clh$ and $\clk$, we denote $\clb(\clh, \clk)$ (and $\clb(\clh)$ if $\clk = \clh$) as the space of all bounded linear operators from $\clh$ to $\clk$. The symbol $H^\infty_{\clb(\C^{p'}, \C^{p+1})}$ denotes the set of $\clb(\C^{p'}, \C^{p+1})$-valued bounded analytic functions on $\D$, where a function $\Phi \in H^\infty_{\clb(\C^{p'}, \C^{p+1})}$ is \textit{inner} if
\[
\Phi(z)^* \Phi(z) = I_{\C^{p'}},
\]
for a.e. $z \in \T$. Moreover
\[
[F_0,1]_{1\times (p+1)} = [f_1, \ldots, f_p, 1],
\]
is a row vector, and $\clk_\Psi$'s elements are columns.

In fact, we prove Theorem \ref{thm: Intro 1} in a much higher generality (see Theorem \ref{Main_th_N}): we simply replace the backward shift $T_z^*$ by $T_\vp^*$, the adjoint of a Toeplitz operator, where the symbol $\vp \in H^\infty$ is an inner function vanishing at the origin. 

Before proceeding, we need the definition of nearly $T_z^*$-invariant subspaces \cite{Sarason}. A closed subspace $\clm$ of $H^2$ is \textit{nearly $T_z^*$-invariant} if
\[
T_z^* f \in \clm,
\]
for all $f \in \clm \cap z H^2$. Now, we will provide a few comments regarding Theorem \ref{thm: Intro 1}. First, this result deals with linear perturbations, a theory in and of itself that is well-known to be challenging (see the classic \cite{Kato}). Second, this parameterizes all of the invariant subspaces of finite-rank perturbations of $T_z^*$, which was known to be a challenging goal. As far as we know, this is the first result in this direction. Third, we provide this result in the tradition of Hayashi, Hitt, and Sarason's classical argumentation. Indeed, Sarason showed \cite[page 488]{Sarason} that if $\clm$ is a nearly $T_z^*$-invariant subspace then there is a unique function
\begin{equation}\label{eqn: intro g}
g \in \clm \ominus (\clm \cap z H^2),
\end{equation}
such that $\clm$ is invariant under the rank-one perturbation $T_z^* - T_z^* g \otimes g$, that is
\begin{equation}\label{eqn: intro Tz M}
(T_z^* - T_z^* g \otimes g) \clm \subseteq \clm.
\end{equation}
Sarason's observation partly motivated us to attempt full-length representations of invariant subspaces of finite-rank perturbations, leading to the formation of Theorem \ref{thm: Intro 1}. At the very start of this section, this point was briefly mentioned. Finally, we recall Hitt's theory \cite{Hitt, Sarason} on almost $T_z^*$-invariant subspaces. Every nearly $T_z^*$-invariant subspace $\clm \subseteq H^2$ admits the representation
\[
\clm = g H^2,
\]
where $g$ is as in \eqref{eqn: intro g}. Undoubtedly, the result in Theorem \ref{thm: Intro 1} is in line with the same fundamentals. As previously said, in the present case, we have partially enacted the core ideas that were put forward by Hayashi, Hitt, and Sarason a few decades ago.

Now we turn to almost invariant subspaces for $T_z^*$. In the setting of the backward shift operator, this was introduced by Chalendar, Gallardo-Guti\'{e}rrez, and Partington \cite{CGP}. Let $T \in \clb(\clh)$. A closed subspace $\clm \subseteq \clh$ is \textit{almost invariant} for $T$ if there exists a finite-dimensional subspace $\clf\subseteq \clh$ such that
\[
T \clm \subseteq \clm \oplus \clf.
\]
In this case, the \textit{defect} of the space $\clm$ is the lowest possible dimension of such a space $\clf$. A special version of Theorem \ref{th-almost} suggests a link between almost invariant subspaces for $T_z^*$ and invariant subspaces of finite-rank perturbations of $T_z^*$:

\begin{theorem}
Let $\clm$ be a closed subspace of $H^2$. If $\clm$ is invariant under
\[
T_z^* - \sum_{i=1}^{m} v_i \otimes u_i,
\]
then $\clm$ is almost invariant under $T_z^*$ with defect at most $m$. Conversely, if $\clm $ is almost invariant under $T_z^*$ with defect $m$, then $\clm$ is invariant under the rank-$t$ perturbation
\[
T_z^* - \sum_{i=1}^{m} f_i \otimes T_z f_i,
\]
where $\{f_i:1\leq i\leq m\}$ is an orthonormal basis of the $m$ dimensional defect space and $t \leq m$.
\end{theorem}

In particular, this and Theorem \ref{thm: Intro 1} yield a complete characterization of almost invariant subspaces for $T_z^*$ (see Corollary \ref{al-crlre}):

\begin{theorem}
Let $\clm \subseteq H^2$ be a closed subspace. If $\clm$ is almost invariant under $T_z^*$, then there exist a natural number $p' \leq p+1$ and an inner function $\Psi \in H^\infty_{\clb(\C^{p'}, \C^{p+1})}$ such that
\begin{equation}\label{eqn: intro IS rep 1A}
\clm = [F_0,1]_{1\times (p+1)} \clk_\Psi,
\end{equation}
where $F_0= [f_1, \ldots, f_p]$, $\{f_i\}_{i=1}^p$ is an orthonormal basis for $\text{span}\{P_\clm T_z u_i: i=1, \ldots, m\}$, and $\{u_i\}_{i=1}^m$ is an orthonormal basis for the defect space. Moreover
\begin{equation}\label{eqn: intro IS rep 2A}
\mathcal{K}_\Psi = \{(F, f_0)\in H^2_{\C^p} \oplus H^2 : F_0 F + f_0 \in \clm \textit{ and } \|F_0 F + f_0\|^2 = \|{F}\|^2 +\|{f_0}\|^2\},
\end{equation}
and
\begin{equation}\label{eqn: intro IS rep 3A}
\mathcal{J}(k_1, \ldots ,k_{p+1})= f_1k_1+ \cdots +f_pk_p +k_{p+1},
\end{equation}
for all $(k_1, \ldots ,k_{p+1}) \in \clk_\Psi$, defines a unitary operator $\mathcal{J} : \clk_\Psi \to \clm$. Conversely, if $\clm$ has the representation \eqref{eqn: intro IS rep 1A} for some orthonormal set of functions $\{f_i\}_{i=1}^p$ in $\clm$, an inner function $\Psi \in H^\infty_{\clb(\C^{p^\prime},\C^{p+1})}$ with $\clk_\Psi$ as in \eqref{eqn: intro IS rep 2A}, and a unitary $\clj$ as in \eqref{eqn: intro IS rep 3A}, then $\clm$ is almost invariant under $T_z^*$.
\end{theorem}

It should be noted that Chalendar, Gallardo-Guti\'{e}rrez, and Partington were the first to classify almost invariant subspaces under $T_z^*$ \cite[Corollary 3.4]{CGP}. However, the above result differs significantly from \cite{CGP}, which compares nearly $T_z^*$-invariant subspaces of finite defect indices with almost invariant subspaces. Here, we connect them with finite-rank perturbations of $T_z^*$. This, in particular, yields more concrete representations of the model space $\clk_\Psi$. See the end of Section \ref{sec: almost IS} for more comments.

Now we turn to a refined version of nearly $T_z^*$-invariant subspaces. The motivation for this arises partly from the matching theory of finite-rank perturbations of $T_\vp^*$ as well as the potential of the possible theory of kernels of finite-rank perturbations of Toeplitz operators, which we will elaborate on shortly. We start by setting up the notations. Let $\{z_k\}_{k=1}^n$ be a sequence of points (possibly repeated) in $\D$, and assume that
\[
z_j = 0,
\]
for some $j=1, \ldots, n$. Define the corresponding Blaschke product $B_n$ by
\[
B_n(z)=\prod_{k=1}^n\dfrac{z_k -z}{1-\bar{z}_k z},
\]
for all $z \in \D$. Clearly, $B_n \in H^\infty$ is an inner function, and
\[
B_n(0) = 0.
\]
We will consider this Blaschke product for the rest of the discussion in this section.

\begin{definition}\label{def: nearly B}
A closed subspace $\clm \subseteq H^2$ is said to be nearly $T^*_{z,B_n}$-invariant if
\[
T_z^* f \in \clm,
\]
for all
\[
f\in \clm \cap B_n H^2.
\]
\end{definition}

Evidently, nearly $T^*_{z,B_1}$-invariant subspaces are precisely nearly $T_z^*$-invariant subspaces. The following provides a complete description of nearly $T^*_{z,B_n}$-invariant subspaces of $H^2$ (for more details, see Theorem \ref{B-inv-thm}):

\begin{theorem}
Let $\clm$ be a nontrivial closed subspace of $H^2$. Then $\clm$ is nearly $T^*_{z,B_n}$-invariant if and only if there exist a natural number $r' \leq r\leq n$ and an inner function $\Theta \in H^\infty_{\clb(\C^{r'},\C^r)}$ with $\Theta(0)=0$ such that
\[
\clm = [g_1,\ldots , g_r]_{1\times r} \clk_\Theta,
\]
where $\{g_i\}_{i=1}^r$ is an orthonormal basis for $\clm \ominus (\clm \cap B_n \hdcc)$, and
\[
\clj(k_1, \ldots ,k_r) = g_1k_1 + \cdots + g_r k_r,
\]
for all $(k_1, \ldots ,k_r) \in \clk_\Theta$, defines a unitary operator $\clj: \clk_\Theta \raro \clm$.
\end{theorem}

Denote by $L^\infty$ the space of all essentially bounded functions on $\T$. For each $\vp \in L^\infty$, we follow the standard notation of the \textit{Toeplitz operator} as $T_\vp$. Then
\begin{equation}\label{eqn: TO}
T_\vp f = P_{H^2}(\vp f),
\end{equation}
for all $f\in H^2$, where $P_{H^2}$ denotes the orthogonal projection (or, the Szeg\"{o} projection) from $L^2$ (the space of all square-integrable functions on $\T$) onto $H^2$. Kernels of Toeplitz operators are crucial in the context of examples of nearly $T_z^*$-invariant subspaces. We refer to \cite{Partington, LP, Sarason 1} for information regarding the kernels of Toeplitz operators and nearly $T_z^*$-invariant subspaces of $H^2$, and \cite{HS} in the context of $H^p$-subspaces.

In a comparable manner, we provide examples of nearly $T^*_{z,B_n}$-invariant subspaces. Curiously, certain dedicated finite-rank perturbations of Toeplitz operators result in nontrivial examples of $T^*_{z,B_n}$-invariant subspaces. Pick an arbitrary orthonormal basis $\{f_i\}_{i=1}^n$ of $\clk_{B_n}$. In Theorem \ref{thm: finite-rank pert TO}, we prove that
\[
\ker \Big(T_\vp + \sum_{i=1}^{n} f_i \otimes T_z^* f_i\Big),
\]
is nearly $T^*_{z,B_n}$-invariant subspace. We are of the opinion that this space is significant.

We remark that Liang and Partington \cite{LP} also examined finite-rank perturbations of Toeplitz operators, providing examples of subspaces that are nearly $T_z^*$-invariant but have finite defects. They impose restrictions on the Toeplitz operators while leaving the finite-rank operators unrestricted. In this context, we also refer to \cite{Fricain}. However, our perspectives and objectives in this paper are different.

Next, we consider representations of invariant subspaces of finite-rank perturbations of $T_z$. We recall a general fact from elementary functional analysis. Let $\cls$ be a closed subspace of $\clh$ and let $T \in \clb(\clh)$. Then $\cls$ is invariant under $T$ if and only if $\cls^\perp$ is invariant under $T^*$. In particular, representations of $T$-invariant subspaces also yield representations of $T^*$-invariant subspaces. Nevertheless, the explicit feature of the invariant property (if any) may not be maintained when representations are transmitted from one to another, particularly in a concrete situation such as ours. Here, we are talking about the representation of invariant subspaces of finite-rank perturbations $T_z$, as we are already aware of the same for finite-rank perturbations of $T_z^*$.

We identify a class of perturbations of $T_z$ for which one can say more definitively about the lattices of invariant subspaces (see Theorem \ref{thm: IS of Tz pert}):

\begin{theorem}
Let $\clm \subseteq H^2$ be a closed subspace. Suppose $\clm$ is an invariant under $T_z - \sum_{i=1}^{m} v_i \otimes u_i$. Define
\[
\cll: = span\{P_{\clm^\perp} v_1, \ldots, P_{\clm^\perp} v_m\}.
\]
Assume that $\{\vp_i\}_{i=1}^p \subseteq H^\infty$ is an orthonormal basis for $\cll$. Then there exists an inner function $\Psi \in H^\infty_{\clb(\C^{p'}, \C^{p+1})}$ for some $p' \leq p+1$ such that
\[
\clm = \{g\in H^2 : (T_{\bar{\vp}_1} g, \ldots, T_{\bar{\vp}_p} g, g) \in \Psi H^2_{\C^{p^\prime}}\}.
\]
\end{theorem}

Some of the results stated above are merely one-directional. This is also in keeping with Hitt and Sarason's perspectives on representations of nearly $T_z^*$-invariant subspaces. In some of the instances, we present a converse directional result that is restricted to a particular class of finite-rank perturbations. These perturbations are motivated by Sarason's rank-one perturbation, as pointed out in \eqref{eqn: intro Tz M}. Typically they look like
\begin{equation}\label{eqn: Sarason}
T_z^* - \sum_{i=1}^m T_z^* f_i \otimes f_i, \text{ or } T_z^* - \sum_{i=1}^m f_i \otimes T_z^* f_i.
\end{equation}
We call them \textit{Sarason-type perturbations}. 

The remaining sections of the paper are structured in the following manner: Section \ref{sec: pert of TO} is the heart of the paper, which gives invariant subspaces of finite-rank perturbations of $T_\vp^*$ with inner symbols $\vp$ vanishing at zero. In Section \ref{sec: pert of shift}, we apply the results of the preceding section to the particular situation of the backward shift operator. We additionally take into account the Sarason-type perturbations. Section \ref{sec: pert of T-z} deals with invariant subspaces of certain finite-rank perturbations of $T_z$. In Section \ref{sec: almost IS}, we determine the lattice of almost invariant subspaces of $T_z^*$. Section \ref{sec3} introduces the concept of nearly $T^*_{z,B_n}$-invariant subspaces and presents an invariant subspace theorem. The final section, Section \ref{sec4}, is devoted to examples of nearly $T^*_{z,B_n}$-invariant subspaces. These subspaces are precisely the kernels of certain finite-rank perturbations of Toeplitz operators. 
 

\section{Perturbations of $T_\vp^*$}\label{sec: pert of TO}
 

In this section, we aim to represent invariant subspaces of finite-rank perturbations of $T_z^*$. However, our techniques are more refined, allowing us to produce results with a significantly higher level of generality. In the following, we establish the problem within the context of the Toeplitz operator with an inner function symbol that vanishes at the origin. Recall that a function $\vp \in H^\infty$ is \textit{inner} if
\[
|\vp(z)| = 1,
\]
for a.e. $z \in \T$ (in the sense of radial limits of $H^\infty$-functions). Let $\vp \in H^\infty$ be an inner function. In this case, the Toeplitz operator $T_\vp$ (see \eqref{eqn: TO}) turns into an analytic Toeplitz operator with the symbol $\vp \in H^\infty$, therefore simplifying to
\[
T_\vp f =\vp f,
\]
for all $f \in H^2$. We will consider rank-$m$ perturbations of $T_\vp^*$. Let $\{u_i\}_{i=1}^m$ and $\{v_i\}_{i=1}^m$ be arbitrary but fixed orthonormal sets of functions in $ H^2$. Define
\[
\frs = T_\vp^* + \sum_{i=1}^{m} v_i \otimes u_i.
\]
Our goal is to parameterize all the invariant subspaces of $\frs$.

To achieve this, we first recall that an operator $T \in \clb(\clh)$ is said to be in $C_{\cdot 0}$ (in short, $T \in C_{\cdot 0}$) if $T$ is a contraction and
\[
\text{SOT}-\lim_{n \raro \infty} T^{*n} = 0,
\]
that is, ${T^*}^nh\to 0$ as $n\to \infty$ for all $h \in \clh$. The following fact is particularly useful, as it affirms that $C_{\cdot 0}$ is invariant under multiplication by projections with finite-codimenional ranges (see \cite[Lemma 3.3]{BeTi00}): Let $T$ on $\clh$ be a $C_{\cdot 0}$ contraction, and let $\clf$ be a finite-codimensional subspace of $\clh$. If $\dim (\text{ran} (I-T^*T)) <\infty$, then
\begin{equation}\label{eqn: Benhida}
TP_\clf\in C_{\cdot 0}.
\end{equation}
Here, given a closed subspace $\cls$ of $\clh$, we denote by $P_\cls$ the orthogonal projection onto $\cls$.

Returning to our perturbation situation, we first define Sarason-type finite-rank perturbation $\sfrs$ on $H^2$ as (see \eqref{eqn: Sarason})
\[
\sfrs := T_\vp^* - \sum_{i=1}^{m}  T_\vp^* u_i \otimes u_i.
\]
In other words, $\sfrs = T_\vp^*(I - \sum_{i=1}^{m}  u_i \otimes u_i)$. Equivalently
\[
\sfrs f = T_\vp^* f - \sum_{i=1}^{m} \langle f, u_i\rangle T_\vp^*u_i,
\]
for all $f \in H^2$. We claim that
\begin{equation}\label{eqn: C dot zero}
\sfrs^* \in  C_{\cdot 0}.
\end{equation}
This is essentially the result of the fact \eqref{eqn: Benhida}. Indeed, first we observe that
\[
\sfrs = T_\vp^*(I-P_{\clf}),
\]
where $\mathcal{F} = span\{u_1, \ldots, u_m\}$ is a finite-dimensional subspace of $H^2$. As $T_\vp$ is a contraction and $I - P_{\clf}$ is an orthogonal projection, $\sfrs$ is a contraction. Since $T_\vp \in C_{ \cdot0}$ and $\clf$ is finite-dimensional, \eqref{eqn: Benhida} stated above implies that $T_\vp(I - P_{\clf})$ is a $C_{ \cdot0}$ contraction. For each $g \in H^2$ and $n \geq 1$, we compute
\[
\sfrs^n g = (T_\vp^*(I-P_{\clf}))^n g = T_\vp^*([T_\vp (I-P_\clf)]^*)^{n-1}(I-P_{\clf})g.
\]
Then
\[
\|\sfrs^n g\| \leq \|T_\vp^*\| \|([T_\vp (I-P_\clf)]^*)^{n-1}(I-P_{\clf})g\|,
\]
implies that $\sfrs^* \in  C_{\cdot 0}$, thereby proving \eqref{eqn: C dot zero}. We emphasize that the technique above has become typical in showing similar results.

We need to fix some more notations. Given a natural number $p \geq 1$, $H^2_{\C^p}$ will denote the $\C^p$-valued Hardy space over $\D$. It is defined as:
\begin{equation}\label{eqn: Hardy}
H^2_{\C^p} =\Big\{F(z)=\sum_{n= 0}^{\infty} A_nz^n: \|F\|^2= \sum_{n= 0}^{\infty} \|{A_n}\|_{\C^p}^2<\infty, A_n\in\C^p, z\in\D\Big\}.
\end{equation}
Furthermore, in the subsequent discussion, we shall define identification as up to unitary equivalence. We will frequently identify $H^2_{\C^p}$ with $H^2 \otimes \C^p$. The other common identification that will be used throughout is $H^2_{\C^{p+1}}$ and $H^2_{\C^p} \oplus H^2$. The concept of model spaces is an additional component that will be utilized in the contents that follow. Given $p' \leq p$ and an inner function $\Psi \in H^\infty_{\clb(\C^{p^\prime},\C^{p})}$, the quotient space
\begin{equation}\label{eqn: model space}
\clk_\Psi = H^2_{\C^{p}}/ \Psi H^2_{\C^{p'}},
\end{equation}
classifies $(T_z^* \otimes I_{\C^p})$-invariant closed subspaces of $H^2_{\C^p}$. This is the consequence of the celebrated Beurling-Lax theorem \cite[Chapter V. Theorem 3.3]{NaFo70}. The subspace $\clk_\Psi$ is commonly referred to as the \textit{model space} corresponding to the inner function $\Psi$ \cite{Nikolski}.

Now we are in a position to present the description of the invariant subspaces of finite-rank perturbations of $T_\vp^*$. In the following statement, $[f_1, \ldots , f_p,1]_{1\times (p+1)}$ refers to the row vector in the $(p+1)$-copies of $H^2$ (which is identified with $H^2_{\C^{p+1}}$). Moreover, often we will express vectors $G$ in $H^2_{\C^{p+1}}$ as column vectors $[g_1, \ldots, g_{p+1}]^t$, where $g_i \in H^2$ for all $i=1, \ldots, p+1$. With this notation in mind, we get
\[
[f_1, \ldots , f_p,1]_{1\times (p+1)} G = \sum_{i=1}^{p}f_i g_i + g_{p+1},
\]
whenever the sum is defined.

\begin{theorem}\label{Main_th_N}
Let $\vp \in H^\infty$ be an inner function, and let $\clm$ be a closed subspace of $H^2$. Suppose $\vp(0)=0$. If $\clm \subseteq H^2$ is an invariant subspace of $T_\vp^* + \sum_{i=1}^{m} v_i \otimes u_i$, then there exists a $(T_\vp^* \otimes I_{\mathbb{C}^{p+1}})$-invariant closed subspace $\clk \subseteq H^2_{\C^{p+1}}$ such that
\[
\clm = [F_0,1]_{1\times (p+1)} \clk,
\]
where $F_0 = [f_1, \ldots , f_p]_{1\times p}$, and $\{f_i\}_{i=1}^p$ is an orthonormal basis for
\[
\text{span}\{P_\clm u_1, \ldots ,P_\clm u_m\}.
\]
Moreover, $\clk$ can be represented as
\[
\mathcal{K} = \{(F, f_0)\in H^2_{\C^p} \oplus H^2 : F_0 F+f_0 \in \clm \textit{ and } \|F\|^2 +\|f_0\|^2 = \|F_0 F+f_0\|^2\}.
\]
\end{theorem}
\begin{proof}
We open by defining the subspace
\[
\clw : = span\{P_\clm u_1, \ldots ,P_\clm u_m\},
\]
and noting that the following inequality occurs within the context of the theorem's statement:
\[
p:=\dim \clw \leq m = \text{rank }(\sum_{i=1}^{m} v_i \otimes u_i).
\]
Following our convention, we set $\frs = T_\vp^* + \sum_{i=1}^{m} v_i \otimes u_i$ and $\sfrs= T_\vp^* - \sum_{i=1}^{m}  T_\vp^* u_i \otimes u_i$. We know that $\clm$ is invariant under $\frs$. If $\clw = \{0\}$, then $\{u_1, \ldots ,u_m\} \perp \clm$, which implies that $\clm$ is $T_\vp^*$ invariant. In this case, we simply write $\clm = [1] \clk$, where $\clk = \clm$. Next, we assume that $\clw$ is nontrivial:
\[
\clw \neq \{0\}.
\]
Then $1 \leq p \leq m$. By noting that $\clw \subseteq \clm$, we decompose $\clm$ as
\[
\clm = \clw \oplus (\clm \cap \clw^\perp).
\]
Pick an orthonormal basis $\{f_i\}_{i=1}^p$ for $\clw$, and fix $f \in \clm$. There exists $\{a_{0j}\}_{j=1}^p \subseteq \C$ such that
\[
 P_{\clw} f = a_{01}f_1 + \cdots +a_{0p}f_p.
\]
We write
\[
 P_{\clw} f = F_0 A_0,
\]
where
\[
A_0= [a_{01}, \ldots ,a_{0p}]^t \in \C^p,
\]
and
\[
F_0 = [f_1, \ldots, f_p]_{1\times p},
\]
denotes the row vector. Since $f = P_{\clw} f \oplus  P_{\clm \cap \clw^\perp} f$, it follows that
\begin{equation}\label{eqn: f = FA+g}
f = F_0 A_0 \oplus g_1,
\end{equation}
where
\[
g_1:= P_{\clm \cap \clw^\perp} f \in \clm \cap \clw^\perp.
\]
Moreover
\[
\|{f}\|^2 = \|{A_0}\|^2 +\|{g_1}\|^2 .
\]
Now we proceed to simplify the function $g_1$.	As $\frs \clm \subseteq \clm$, it follows that
\[
l_1 := \frs g_1 \in \clm.
\]
We also compute
\[
\begin{split}
l_1 & = (T_\vp^* - \sum_{i=1}^{m} v_i \otimes u_i) g_1
\\
& = T_\vp^* g_1 - \sum_{i=1}^{m} \langle g_1, u_i\rangle v_i
\\
& = T_\vp^* g_1,
\end{split}
\]
as $g_1 \in \clm \cap \clw^\perp$ and $\clw = \text{span}\{P_\clm u_j: j=1, \ldots, m\}$. Since $T_\vp T_\vp^* = I - P_{\clk_\vp}$, we conclude that $T_\vp l_1  = T_\vp T_\vp^* g_1 = g_1 -P_{\clk_\vp} g_1$. This readily implies that (note that $T_\vp l_1 = \vp l_1$)
\[
g_1 = \vp l_1 \oplus P_{\clk_\vp} g_1,
\]
and then $f = F_0 A_0 \oplus (l_1 \vp + P_{\clk_\vp} g_1)$. Also, we have
\[
\|{g_1}\|^2 = \|{l_1}\|^2 + \|P_{\clk_\vp} g_1\|^2.
\]
By summarizing all the above observations, we conclude that
\[
\begin{cases}
f = F_0 A_0 \oplus l_1\vp \oplus P_{\clk_\vp} g_1
\\
\|{f}\|^2 = \|{A_0}\|^2 +\|{l_1}\|^2 + \|P_{\clk_\vp} g_1\|^2,
\end{cases}
\]
and
\[
l_1 = (\frs P_{\clm \cap \clw^\perp}) f.
\]
We now employ the technique of representing $f \in \clm$ outlined above to represent $l_1 \in \clm$ (which will lead to the induction steps). Therefore, we first decompose
\[
l_1 = F_0A_1 \oplus g_2,
\]
where $g_2 = P_{\clm \cap \clw^\perp} l_1 \in \clm \cap \clw^\perp$, $A_1 = [a_{11}, \ldots ,a_{1p}]^t \in \C^p$, and
\[
\|{l_1}\|^2 = \|{A_1}\|^2 +\|{g_2}\|^2.
\]
A computation similar to the one above implies that
\[
g_2 = l_2\vp \oplus P_{\clk_\vp}g_2,
\]
and
\[
\|{g_2}\|^2 = \|{l_2}\|^2 + \| P_{\clk_\vp}g_2\|^2,
\]
where
\[
l_2:=\frs g_2 \in \clm.
\]
Therefore, after these two steps, we have
\[
\begin{cases}
f = F_0 (A_0 +A_1\vp) + l_2\vp^2 +(P_{\clk_\vp} g_1 + (P_{\clk_\vp}g_2)\vp)
\\
\|{f}\|^2 = \|{A_0}\|^2 +\|{A_1}\|^2 +\|{l_2}\|^2 +\| P_{\clk_\vp} g_1\|^2 +\| P_{\clk_\vp} g_2\|^2.
\end{cases}
\]
We also have $l_2 = \frs g_2 = \frs P_{\clm \cap \clw^\perp} l_1$, and then
\[
l_2 = \frs P_{\clm \cap \clw^\perp}  \frs g_1 = \frs P_{\clm \cap \clw^\perp} \frs P_{\clm \cap \clw^\perp} f,
\]
that is
\[
l_2 = (\frs P_{\clm \cap \clw^\perp})^2 f.
\]
By continuing this process, we obtain
\begin{equation}\label{norm_n2}
\begin{cases}
f= F_0 (\sum_{j=0}^n A_j z^j) + l_{n+1}\vp^{n+1} + \sum_{j=1}^{n+1} (P_{\clk_\vp} g_j) \vp^{j-1}
\\
\|{f}\|^2 = \sum_{i=0}^{n} \|{A_i}\|^2 + \|{l_{n+1}}\|^2 + \sum_{j=1}^{n+1}\|P_{\clk_\vp} g_j\|^2,
\end{cases}		
\end{equation}
and
\[
l_n = (\frs P_{\clm \cap \clw^\perp})^n f,
\]
and $\{A_i\}_{i=1}^n \subseteq \C^p$ for all $n \geq 1$. On the other hand, $P_{\clm \cap \clw^\perp} = P_\clm P_{\clm \cap \clw^\perp}$ and $P_{\clm} u_i \in \clw$ implies that
\[
(v_i \otimes u_i) P_{\clm \cap \clw^\perp} = 0 = (T_\vp^* u_i \otimes u_i) P_{\clm \cap \clw^\perp},
\]
for all $i=1, \ldots,m$, and hence
\[
\frs P_{\clm \cap \clw^\perp} = \sfrs P_{\clm \cap \clw^\perp} (= T_\vp^* P_{\clm \cap \clw^\perp}).
\]
It follows that
\[
l_{n} = (\sfrs P_{\clm \cap \clw^\perp})^{n} f,
\]
for all $n \geq 1$. By \eqref{eqn: C dot zero}, we know, in particular, that $\sfrs^* \in C_{\cdot 0}$ and hence \eqref{eqn: Benhida} implies
\[
\sfrs^* (I - P_\clw) \in C_{\cdot 0}.
\]
Also, it is clear that $P_{\clm\cap \clw^\perp}h = (I-P_\clw)h$ for all $h\in \clm$. As
\[
\|{l_{n+1}}\| = \|{(\sfrs P_{\clm \cap \clw^\perp})^{n+1} f}\| \leq \|{\sfrs}\|\|{(P_{\clm \cap \clw^\perp} \sfrs)^n P_{\clm \cap \clw^\perp}f}\|,
\]
for all $n \geq 1$, we conclude that $\|{l_{n+1}}\| \to 0$ as $n \to \infty$. By \eqref{norm_n2}, we know that $\|{f}\|^2 = \sum_{i=0}^{n} \|{A_i}\|^2 + \|{l_{n+1}}\|^2 + \sum_{j=1}^{n+1}\|P_{\clk_\vp} g_{j}\|^2$, and hence
\[
\max\Big\{\sum_{i=0}^{n} \|{A_i}\|^2, \sum_{j=1}^{n+1} \|P_{\clk_\vp} g_j\|^2\Big\} \leq \|{f}\|^2,
\]
yielding
\[
\max\Big\{\sum_{i=0}^\infty \|{A_i}\|^2, \sum_{j=1}^\infty \|P_{\clk_\vp}g_j\|^2\Big\} \leq \|{f}\|^2.
\]
Define
\[
F(z) = \sum_{i=0}^{\infty}A_i\vp^i,
\]
and
\[
f_0(z)= \sum_{j=1}^{\infty}(P_{\clk_\vp} g_j)\vp^{j-1},
\]
for all $z \in \D$. We claim that $F \in H^2_{\C^p}$ and $f_0 \in H^2$. Observe that $\ker T_\vp^* = \clk_\vp$. Because $T_\vp$ is a pure isometry (that is, a shift operator), by the von Neumann-Wold decomposition theorem \cite[Chapter I. Theorem 1.1]{NaFo70}, we have the orthogonal decomposition
\[
H^2 = \clk_\vp \oplus \vp \clk_\vp \oplus \vp^2 \clk_\vp \oplus \cdots.
\]
Since $\sum_{j=0}^\infty \|P_{\clk_\vp}g_j\|^2 < \infty$, in view of the above decomposition, we conclude that $f_0 \in H^2$. For the $\C^p$-valued function $F$, we set
\[
T_\Phi = T_\vp \otimes I_{\mathbb{C}^p}.
\]
Since $\vp(0)=0$ implies $\mathbb{C} \subseteq \clk_\vp$, we infer that
\[
\mathbb{C}^p \cong \mathbb{C} \otimes \mathbb{C}^p \subseteq \clk_\Phi,
\]
where $\clk_\Phi = \ker T_\Phi^*$. Again, since $T_\Phi$ on $H^2 \otimes \mathbb{C}^p$ is a pure isometry (as $T_\Phi^{*n} = T^{*n}_\vp \otimes I_{\C^p} \raro 0$ in SOT), we have the von Neumann-Wold decomposition as
\[
H^2 \otimes \mathbb{C}^p =  \clk_\Phi \oplus \Phi \clk_\Phi \oplus \Phi^2 \clk_\Phi \oplus \cdots.
\]
Because $\mathbb{C}^p \subseteq \clk_\Phi$ and $\sum_{i=0}^\infty \|{A_i}\|^2 < \infty$, we conclude that $F \in H^2_{\C^p}$. Next, for each $n \geq 1$, we consider the $n$-th partial sums
\[
K_{n} = \sum_{i=0}^{n}A_i\vp^i,
\]
and
\[
k_n = \sum_{i=0}^{n}P_{\clk_\vp}(g_{i+1})\vp^i.
\]
The representation of $f$ in \eqref{norm_n2} then implies that
\[
\|{f- F_0K_n-k_n}\|_2 = \|{l_{n+1}}\|_2 \to 0,
\]
in $H^2$, and hence
\[
\|{f- F_0K_n-k_n}\|_1 \leq \|{f- F_0K_n-k_n}\|_2 \raro 0,
\]
in $H^1$. Recall that $H^1$ is the Banach space of analytic functions on $\D$ defined by \cite[Volume 1. Chapter 3]{Nikolski}
\[
H^1 = \{f \in \text{Hol}(\D): \|f\|:= \sup_{0 \leq r < 1} \int_{0}^{2\pi} |f(r e^{i\theta})|\frac{d\theta}{2 \pi} < \infty\}.
\]
Also, we compute
\begin{align*}
\|{F_0 F + f_0 - F_0 K_n - k_n}\|_1 &\leq \|{F_0 F - F_0 K_n}\|_1 +\|{f_0-k_n}\|_1\\
& \leq	\|{F_0}\|_2\|{F - K_n}\|_2 + \|{f_0 - k_n}\|_2
\\
& \longrightarrow 0.
\end{align*}
This along with $\|{f- F_0K_n-k_n}\|_1 \raro 0$ implies that $f =F_0 F +f_0$ in $H^1$. But $f\in H^2$, and then, we finally get the desired representation of $f \in \clm$ as
\begin{equation}\label{eqn: f rep}
f = F_0 F + f_0,
\end{equation}
along with the norm condition
\begin{equation}\label{eqn: f rep norm}
\|{f}\|^2 = \|F\|^2 +\|f_0\|^2.
\end{equation}
We now prove that the representation of $f$ in \eqref{eqn: f rep} under the condition \eqref{eqn: f rep norm} is unique. Suppose $f = F_0 \tilde{F} + \tilde{f}_0 = F_0 \hat{F} + \hat{f}_0$ for some $\tilde{F}, \hat{F} \in H^2_{\C^p}$ and $\tilde{f}_0, \hat{f}_0 \in H^2$. Then
\[
F_0(\tilde{F} - \hat{F}) + (\tilde{f}_0 - \hat{f}_0) = 0 \in \clm,
\]
and by \eqref{eqn: f rep norm}, it follows that
\[
\|\tilde{F} - \hat{F}\|^2 + \|\tilde{f}_0 - \hat{f}_0\|^2 = 0.
\]
Therefore, $\tilde{F} = \hat{F}$ and $\tilde{f}_0 = \hat{f}_0$, and hence, under the norm condition \eqref{eqn: f rep norm}, the representations of elements of $\clm$ as in \eqref{eqn: f rep} are unique. In view of this, now we define the closed subspace $\clk$ of $H^2_{\C^p} \oplus H^2 \cong H^2_{\C^{p+1}}$ as
\[
\mathcal{K}:= \{(F,f_0)\in H^2_{\C^p} \oplus H^2 : F_0 F + f_0 \in \clm \textit{ and } \|F\|^2 +\|{f_0}\|^2 = \|F_0 F + f_0\|^2\}.
\]
We claim that $\clk$ is invariant under $(T_\vp \otimes I_{\C^{p+1}})^*$. To see this, pick $(F, f_0)\in \clk$. By the construction of $\clk$, there exists $g \in \clm$ such that
\[
g = F_0 F +f_0.
\]
By the uniqueness part of $g$ as in \eqref{eqn: f rep} under the condition \eqref{eqn: f rep norm}, it follows that $F(z) = \sum_{i=0}^{\infty}A_i\vp^i$, and $f_0(z)= \sum_{j=1}^{\infty}(P_{\clk_\vp} g_j)\vp^{j-1}$. In the following, we shall identify $H^2$ functions with their radial limit representations. From this point of view, it follows that
\[
\vp \bar{\vp} =1,
\]
and hence
\begin{align*}
f &= F_0 F + f_0
\\
& = F_0 F(0) + F_0(F - F(0)) + (f_0 - P_{\clk_\vp} f_0) + P_{\clk_\vp} f_0
\\
& = F_0 F(0) + F_0\vp\bar{\vp}\left(F - F(0)\right)+ \vp\bar{\vp}\left(f_0-P_{\clk_\vp}f_0 \right) + P_{\clk_\vp}f_0
\\
& = F_0 F(0) + \vp (F_0 \bar{\vp}\left(F - F(0)\right)+ \bar{\vp}\left(f_0-P_{\clk_\vp}f_0 \right)) + P_{\clk_\vp}f_0.
\end{align*}
From the representations of $F$ and $f_0$, it is easy to see that $T_\Phi^* F = \bar{\vp}(F - F(0))$ and $T_\vp^* f_0 = \bar{\vp}\left(f_0-P_{\clk_\vp}f_0\right)$, and hence
\[
f = F_0 F(0) + \vp (F_0 T_\Phi^* F + T_\vp^* f_0) + P_{\clk_\vp} f_0.
\]
But $F(0) = A_0$, and then
\[
f = F_0 A_0 + \vp (F_0 T_\Phi^* F + T_\vp^* f_0) + P_{\clk_\vp} f_0.
\]
Recall from the above construction that $f = F_0 A_0 \oplus (l_1 \vp + P_{\clk_\vp} g_1)$. Since
\[
P_{\clk_\vp} f_0 = P_{\clk_\vp} (\sum_{j=1}^{\infty}(P_{\clk_\vp} g_j)\vp^{j-1}) = P_{\clk_\vp} g_1,
\]
we conclude that
\[
f = F_0 A_0 \oplus (l_1 \vp + P_{\clk_\vp} f_0).
\]
Comparing this with the above representation of $f$, we finally obtain (recall that $\vp\bar{\vp} = 1$) $F_0 T_\Phi^* F + T_\vp^* f_0 = l_1 \in \clm$. In other words
\[
((T_\vp^* \otimes I_{\C^p}) F, T_\vp^* f_0) = (T_\Phi^* F, T_\vp^* f_0)\in \clk,
\]
and hence, $\clk$ is invariant under $T_\vp^* \otimes I_{\C^{p+1}}$. This completes the proof of the theorem.
\end{proof}

As we previously mentioned, complete description of the lattices of invariant subspaces of natural operators is a frequently difficult task and a rare occurrence. The answers often offer useful information on a variety of subjects (cf. the Beurling theorem \cite{Beurling} and its applications \cite{NaFo70, Nikolski}). At this moment, it is unclear how far the above result extends in terms of application. However, we will continue our discussion in light of the above results. First, we apply it to the case $\theta(z) = z$ for all $z \in \D$. We already pointed out that the $(T_z^* \otimes I_{\C^p})$-invariant subspaces are precisely model spaces (see \eqref{eqn: model space}). Therefore, we have:

\begin{corollary}\label{cor: T_z per inv sub}
Let $\clm$ be a closed subspace of $H^2$. If $\clm$ is invariant under $T_z^* + \sum_{i=1}^{m} v_i \otimes u_i$, then there exist $p' \leq p+1$ and an inner function $\Psi \in H^\infty_{\clb(\C^{p'}, \C^{p+1})}$ such that
\[
\clm = [F_0,1]_{1\times (p+1)} \clk_\Psi,
\]
where $F_0 = [f_1, \ldots, f_p]$, and $\{f_i\}_{i=1}^p$ is an orthonormal basis for $\text{span}\{P_\clm u_1, \ldots ,P_\clm u_m\}$. Moreover
\[
\mathcal{K}_\Psi = \{(F,f_0)\in H^2_{\C^p} \oplus H^2 : F_0 F + f_0 \in \clm \textit{ and } \|F_0 F + f_0\|^2 = \|F\|^2 +\|f_0\|^2\}.
\]
\end{corollary}

In the following section, we will also talk about the opposite direction of the above result.

\section{Perturbations of $T_z^*$}\label{sec: pert of shift}

The setting of Theorem \ref{Main_th_N} is exceedingly broad, embracing invariant subspaces of finite-rank perturbations of $T_\vp^*$. This one-sided result fits Hitt and Sarason's perspectives. Indeed, their results were only one-directional. In this section, we depart from this pattern and try to explain a converse direction. Our converse, however, makes better sense for the particular case of the special inner function: $\vp(z) = z$, $z \in \D$.

Thus, we will follow the setting of Corollary \ref{cor: T_z per inv sub} and subsequently keep on assuming that $\{u_i\}_{i=1}^m$ and $\{v_i\}_{i=1}^m$ are orthonormal sets of vectors in $H^2$. To address the reverse direction, we first note a consequence of Corollary \ref{cor: T_z per inv sub}: There exists a unitary operator $\mathcal{J} : \clk \to \clm$,  such that
\[
\mathcal{J}(k_1, \ldots ,k_{p+1})= f_1k_1+ \cdots +f_pk_p +k_{p+1},
\]
for all $(k_1, \ldots ,k_{p+1}) \in \clk_\Psi$. This readily follows from the fact that an element $(F, f_0) \in \clk_\Psi$ can be written as $(F, f_0) = (k_1, \ldots, k_p, k_{p+1})$ (here we are rewriting $f_0 = k_{p+1}$ for notational simplicity), so that
\[
F_0 F + f_0 = f_1k_1+ \cdots +f_pk_p +k_{p+1}.
\]
We also have the identity that
\[
\|F_0 F + f_0\|^2 = \|F\|^2 +\|f_0\|^2.
\]

We fix the rank-$m$ perturbation $T_z^* - \sum_{i=1}^{m} v_i \otimes u_i$, and set
\[
\clw : = span\{P_\clm u_1, \ldots ,P_\clm u_m\}.
\]
Given a natural number $p$ (the number $p$ will be clear from the context), recall the vector-valued Hardy space identification
\[
H^2_{\C^p} \oplus H^2 \cong H^2_{\C^{p+1}}.
\]
We now move on to representing constant functions in $H^2_{\C^{p+1}}$ (like the constant function $1$ in $H^2$). Clearly, a total of $(p+1)$ basic functions of this type exist. We represent them as $\{e_i\}_{i=1}^{p+1}$. In other words, for every $i=1, \ldots, p+1$, we have
\[
e_i = (0, \ldots, 0,1,0, \ldots, 0),
\]
with $1$ at the $i$-th position and zero elsewhere. We follow the outcome of Corollary \ref{cor: T_z per inv sub}. For $F = e_i$ with $i=1, \ldots, p$, it follows that
\[
F_0 e_i + 0 = f_i \in \clm,
\]
as $\{f_j\}_{j=1}^p \subseteq \clm$. This and
\[
\|e_i\|^2 + \|0\|^2 = \|(e_i,0)\|^2,
\]
ensure that
\begin{equation}\label{eqn: e_i in K_psi}
\{e_j\}_{j=1}^p \subseteq \clk_\Psi.
\end{equation}
In this context, the first part of the result below is a particular application of Corollary \ref{cor: T_z per inv sub}.

\begin{theorem}\label{Main-th}
Let $\clm$ be a closed subspace of $H^2$. If $\clm$ is invariant under $T_z^* - \sum_{i=1}^{m} T_z^* u_i \otimes u_i$, then there exist a natural number $p' \leq p+1$ and an inner function $\Psi \in H^\infty_{\mathcal{B}(\C^{p^\prime},\C^{p+1})}$ such that
\begin{equation}\label{eqn: IS rep 1}
\clm = [F_0,1]_{1\times (p+1)} \clk_\Psi,
\end{equation}
where $F_0 = [f_1, \ldots, f_p]$ and $\{f_i\}_{i=1}^p$ is an orthonormal basis for $\clw$. Moreover
\begin{equation}\label{eqn: IS rep 2}
\mathcal{K}_\Psi = \{(F, f_0)\in H^2_{\C^p} \oplus H^2 : F_0 F + f_0 \in \clm \textit{ and } \|F_0 F + f_0\|^2 = \|{F}\|^2 +\|{f_0}\|^2\},
\end{equation}
and there exists a unitary operator $\mathcal{J} : \clk_\Psi \to \clm$ such that
\begin{equation}\label{eqn: IS rep 3}
\mathcal{J}(k_1, \ldots ,k_{p+1})= f_1k_1+ \cdots +f_pk_p +k_{p+1},
\end{equation}
for all $(k_1, \ldots ,k_{p+1}) \in \clk_\Psi$. Conversely, if $\clm$ has the representation \eqref{eqn: IS rep 1} for some orthonormal set of vectors $\{f_i\}_{i=1}^p$ in $\clm$, an inner function $\Psi \in H^\infty_{\mathcal{B}(\C^{p^\prime},\C^{p+1})}$ with $\clk_\Psi$ as in \eqref{eqn: IS rep 2}, and a unitary $\clj$ as in \eqref{eqn: IS rep 3}, then $\clm$ is invariant under $T_z^* - \sum_{i=1}^p T_z^* f_i \otimes f_i$.
\end{theorem}
\begin{proof}
We only need to prove the sufficient part. In this case, note that
\[
\clm = \{f\in H^2: f=F_0 F + f_0 \text{ for some } (F, f_0)\in \clk_\Psi \text{ with } \|f\|^2 = \|F\|^2 +\|{f_0}\|^2 \}.
\]
Given $f\in \clm$, the above representation of $\clm$ implies the existence of $(F, f_0) \in \clk_\Psi$ such that $f = F_0 F + f_0$. If we write $F = (k_0^1, \ldots, k_0^p)^t \in H^2_{\C^p}$, then
\begin{align*}
(T_z^* - \sum_{i=1}^p T_z^* f_i \otimes f_i) f &= [T_z^* - \sum_{i=1}^{p} \langle \cdot , f_i\rangle T_z^*f_i] (F_0 F + f_0)
\\
& = \frac{F_0(z) F(z) -F_0(0) F(0)}z + T_z^* f_0 - \sum_{i=1}^{p} \langle F_0 F + f_0, f_i\rangle T_z^*f_i.
\end{align*}
On the other hand
\[
\clj(k_0^1, \ldots, k_0^p, f_0) = F_0 F + f_0,
\]
and
\[
\clj e_i = f_i,
\]
for all $i=1, \ldots, p$. This conclusion arises immediately from the identity \eqref{eqn: e_i in K_psi}. Since $\clj$ is an isometry (in fact, unitary), for each $i=1, \ldots, p$, it follows that
\[
\begin{split}
\langle F_0 K_0 + f_0 , f_i\rangle & = \langle \clj(k_0^1, \ldots, k_0^p, f_0), \clj e_i \rangle
\\
& = \langle (k_0^1, \ldots, k_0^p, f_0), e_i \rangle
\\
& = k_0^i(0).
\end{split}
\]
Moreover, since $T_z^* F = \frac{F - F(0)}{z}$ and $T_z^* F_0 = \frac{F_0 - F_0(0)}{z}$, we have
\begin{align*}
(T_z^* - \sum_{i=1}^p T_z^* f_i \otimes f_i) f & = F_0 T_z^*(F) + T_z^*(F_0) F(0) + T_z^* f_0 - \sum_{i=1}^{p} k_0^i(0) T_z^*f_i
\\
& = F_0 (T_z^* F) + (T_z^* F_0) F(0) + T_z^* f_0 - (T_z^* F_0)F(0)
\\
& = F_0 (T_z^* F) + T_z^* f_0
\\
& \in \clm,
\end{align*}
proving that $\clm$ is invariant under $T_z^* - \sum_{i=1}^p T_z^* f_i \otimes f_i$. Take note that in the above, we have used the same notation $T_z^*$ to denote the operator $(T_z^* \otimes I_{\C^p})$ on $H^2_{\C^p}$.
\end{proof}

Considering the converse half of the previous statement, it is evident that the theorem concerns the classification of closed subspaces of $H^2$ that are invariant under Sarason-type perturbations of $T_z^*$. Furthermore, in the setting of Corollary \ref{cor: T_z per inv sub}, the same proof of the converse part also shows that $\clm$ is invariant under $T_z^* - \sum_{i=1}^p T_z^* f_i \otimes f_i$. The following alternative can be considered a variation of Corollary \ref{cor: T_z per inv sub}:

\begin{corollary}\label{cor: var T_z per inv sub}
Let $\clm$ be a closed subspace of $H^2$. If $\clm$ is invariant under $T_z^* + \sum_{i=1}^m v_i \otimes u_i$, then there exist $p' \leq p+1$ and an inner function $\Psi \in H^\infty_{\clb(\C^{p'}, \C^{p+1})}$ such that
\[
\clm = [F_0,1]_{1\times (p+1)} \clk_\Psi,
\]
where $\{f_i\}_{i=1}^p$ is an orthonormal basis for $\text{span}\{P_\clm u_1, \ldots ,P_\clm u_m\}$. Moreover, $\clm$ is invariant under $T_z^* - \sum_{i=1}^p T_z^* f_i \otimes f_i$.
\end{corollary}

 

\section{Perturbations of $T_z$}\label{sec: pert of T-z}

Let $T \in \clb(\clh)$ and let $\clm$ be a closed subspace of $\clh$. In the introductory section, we already pointed out the elementary fact that
\[
T\clm \subseteq \clm,
\]
if and only if
\[
T^* \clm^\perp \subseteq \clm^\perp.
\]
However, if $\clm$ is explicit, then it need not be the case for $\clm^\perp$. For instance, the concrete representations of finite-rank perturbations of $T_z^*$-invariant subspaces do not seem to yield concrete representations of finite-rank perturbations of $T_z$-invariant subspaces. In this section, we examine a class of perturbations of $T_z$, and obtain explicit representations of invariant subspaces of the perturbations. As hinted above, the primary method will involve exploiting the adjoint of the perturbation used in Theorem \ref{Main_th_N}.

As usual, we consider orthonormal sets of vectors $\{u_i\}_{i=1}^m$ and $\{v_i\}_{i=1}^m$ in $ H^2$. Given a closed subspace $\clm \subseteq H^2$, we let
\[
\cll : = span\{P_{\clm^\perp} v_1, \ldots, P_{\clm^\perp} v_m\},
\]
and assume that
\[
p = \dim \cll.
\]
In comparison to Theorem \ref{Main_th_N}, the following assumes that there is an orthonormal basis for $\cll$ that is contained in $H^\infty$.

\begin{theorem}\label{thm: IS of Tz pert}
Let $\clm$ be an invariant subspace of $T_z - \sum_{i=1}^{m} v_i \otimes u_i$. Suppose $\{\vp_i\}_{i=1}^p \subseteq H^\infty$ is an orthonormal basis for $\cll$. Then
\[
\clm = \{g\in H^2 : (T_{\bar{\vp}_1} g, \ldots, T_{\bar{\vp}_p} g, g) \in \Psi H^2_{\C^{p^\prime}}\},
\]
for some $p' \leq p+1$ and inner function $\Psi \in H^\infty_{\clb(\C^{p'}, \C^{p+1})}$.
\end{theorem}

\begin{proof}
Note that $\clm^\perp$ is invariant under $T_z^* -\sum_{i=1}^{m} u_i \otimes v_i$. By Theorem \ref{Main_th_N}, there exist $p' \leq p+1$ and an inner function $\Psi \in H^\infty_{\clb(\C^{p'}, \C^{p+1})}$ such that
\[
\clm^\perp = [\vp_1, \ldots, \vp_p, 1]_{1 \times (p+1)} \clk_\Psi.
\]
Let us assume that $g \in H^2$. Then $g \in \clm^\perp$ if and only if $\langle g,f\rangle =0$ for all $f\in \clm^\perp$. By the above representation of $\clm^\perp$, we consider an arbitrary $f$ from $\clm^\perp$ as $f = \sum_{i=1}^{p} \vp_i k_i + k_{p+1}$ for some $(k_1, \ldots, k_{p+1})\in \clk_\Psi$, and compute
\begin{align*}
\langle g,f\rangle & = \langle g, \sum_{i=1}^{p} \vp_i k_i + k_{p+1} \rangle
\\
& = \langle g, \sum_{i=1}^{p} \vp_i k_i \rangle + \langle g, k_{p+1} \rangle
\\
& = \sum_{i=1}^{p} \langle T_{\bar{\vp}_i} g, k_i \rangle + \langle g, k_{p+1} \rangle.
\end{align*}
Therefore, $\langle g,f\rangle =0$ for all $f\in \clm^\perp$ if and only if
\[
(T_{\bar{\vp}_1} g, \ldots, T_{\bar{\vp}_p} g, g) \in \Psi H^2_{\C^{p^\prime}}(\D),
\]
completing the proof of the theorem.
\end{proof}

To formulate a more convenient converse statement, we will now concentrate on Sarason-type perturbations of the forward shift (see \eqref{eqn: Sarason}). Recall the definition of Toeplitz operators (see \eqref{eqn: TO}). The well-known algebraic characterizations of Toeplitz operators state the following: Let $T \in \clb(H^2)$. Then $T$ is Toeplitz if and only if
\begin{equation}\label{B-H}
T_z^*T T_z = T.
\end{equation}

\begin{theorem}
Let $\clm \subseteq H^2$ be a closed subspace. Assume that $\{\vp_i\}_{i=1}^p \subseteq H^\infty$ is an orthonormal basis for $\cll$. If $\clm$ is invariant under $T_z - \sum_{i=1}^{m} T_z^* v_i \otimes v_i$, then there exist $p' \leq p+1$ and an inner function $\Psi \in H^\infty_{\clb(\C^{p'}, \C^{p+1})}$ such that
\[
\clm = \{g\in H^2 : (T_{\bar{\vp}_1} g, \ldots, T_{\bar{\vp}_p} g, g) \in \Psi H^2_{\C^{p^\prime}}\},
\]
and the operator $\mathcal{J} : \clk_\Psi\to \clm^\perp$ defined by
\[
\mathcal{J}(k_1, \ldots ,k_{p+1}) = \vp_1k_1 + \cdots + \vp_pk_p + k_{p+1},
\]
for all $(k_1, \ldots ,k_{p+1}) \in \clk_\Psi$ is unitary. Conversely, if $\clm$ admits the above representation along with the existence of the unitary $\clj$, then $\clm$ is invariant under $T_z - \sum_{i=1}^{p} \vp_i \otimes T_z^* \vp_i$.
\end{theorem}
\begin{proof}
We only have to prove the reverse direction. Note that $T_{\vp_i}^* = T_{\bar{\vp}_i}$ for all $i=1, \ldots, p$, and hence
\[
\mathcal{J}^* h = (T_{\bar{\vp}_1} h, \ldots, T_{\bar{\vp}_p} h, h) \in \clk_\Psi,
\]
for all $h\in \clm^\perp$. Let $g\in \clm$. Then
\[
(T_z - \sum_{i=1}^{p} \vp_i \otimes T_z^* \vp_i)g = T_z g - \tilde{g}
\]
where $\tilde{g} = \sum_{i=1}^{p}\langle g, T_z^* \vp_i\rangle \vp_i$, and also $(T_{\bar{\vp}_1} g, \ldots, T_{\bar{\vp}_p} g, g) \in \Psi H^2_{\C^{p^\prime}}(\D)$. Hence
\begin{align*}
(T_{\bar{\vp}_1}, \ldots, T_{\bar{\vp}_p}, I) (T_z g - \tilde{g}) & = (T_{\bar{\vp}_1}, \ldots, T_{\bar{\vp}_p}, I) T_z g - (T_{\bar{\vp}_1}, \ldots, T_{\bar{\vp}_p}, I) \tilde{g}
\\
& = (T_{\bar{\vp}_1}T_z, \ldots, T_{\bar{\vp}_p}T_z, T_z) g - (T_{\bar{\vp}_1}, \ldots, T_{\bar{\vp}_p}, I) \tilde{g}
\\
& = (T_z T_z^* T_{\bar{\vp}_1}T_z, \ldots, T_z T_z^* T_{\bar{\vp}_p}T_z, T_z) g - (T_{\bar{\vp}_1}, \ldots, T_{\bar{\vp}_p}, I) \tilde{g}
\\
& \quad  + (P_{\C} T_{\bar{\vp}_1}T_z, \ldots, P_{\C} T_{\bar{\vp}_p}T_z, 0) g,
\end{align*}
as $T_z T_z^* + P_\C = I$, where $P_\C$ denotes the orthogonal projection onto the one-dimensional space of constant functions in $H^2$. Since $T_{\bar{\vp}_j}$ is a Toeplitz operator, $T_z^*T_{\bar{\vp}_p}T_z = T_{\bar{\vp}_p}$ for all $j=1, \ldots, p$ (see \eqref{B-H}), and consequently
\begin{align*}
(T_{\bar{\vp}_1}, \ldots, T_{\bar{\vp}_p}, I) (T_z g - \tilde{g}) & = T_z (T_{\bar{\vp}_1}, \ldots, T_{\bar{\vp}_p}, I) g - \mathcal{J}^* (\sum_{i=1}^{p}\langle g , T_z^* \vp_i\rangle \mathcal{J}e_i) + \sum_{j=1}^{p}\langle T_z g, \vp_j \rangle e_j
\\
& = T_z (T_{\bar{\vp}_1}, \ldots, T_{\bar{\vp}_p}, I) g - \sum_{i=1}^{p}\langle g , T_z^* \vp_i\rangle e_i + \sum_{j=1}^{p}\langle g, T_z^* \vp_j \rangle e_j
\\
& = T_z (T_{\bar{\vp}_1}, \ldots, T_{\bar{\vp}_p}, I) g
\\
& \in \Psi H^2_{\C^{p^\prime}}(\D),
\end{align*}
which implies that $\clm$ is invariant under the Sarason-type perturbation $T_z - \sum_{i=1}^{p} \vp_i \otimes T_z^* \vp_i$.
\end{proof}

Now let us reexamine the discussion that took place at the end of Section \ref{sec: pert of shift}. In particular, we continue the results described in Corollary \ref{cor: var T_z per inv sub}. Along this line, we note that the subspace $\clm$ above is also invariant under $T_z - \sum_{i=1}^p \vp_i \otimes T_z^* \vp_i$. To confirm this, we proceed as follows: Pick $f \in \clm^\perp$. By Corollary \ref{cor: var T_z per inv sub}, we have
\[
f = F_0 F + f_0,
\]
where $F = \sum_{i=0}^{\infty} A_i z^i$, $A_i \in \C^p$ for all $i$, and $f_0 = \sum_{i=0}^{\infty} g_{i+1}(0) z^i$. Then
\[
\begin{split}
(T_z^* - \sum_{i=1}^p T_z^* \vp_i \otimes \vp_i) f & = (T_z^* - \sum_{i=1}^p T_z^* \vp_i \otimes \vp_i) (F_0F + f_0)
\\
& = T_z^*(F_0F) + T_z^* f_0 - \sum_{i=1}^p \langle F_0F + f_0, \vp_i \rangle T_z^* \vp_i.
\end{split}
\]
Now assume $F = (k_0^1, \ldots, k_0^p)$, and proceed with the proof of Theorem \ref{Main-th}. We find
\[
\langle F_0 K_0 + f_0 , \vp_i\rangle = k_0^i(0),
\]
and finally
\[
(T_z^* - \sum_{i=1}^p T_z^* \vp_i \otimes \vp_i) f = F_0 (T_z^* F) + T_z^* f_0 \in \clm^\perp,
\]
proves that $\clm^\perp$ is invariant under the Sarason-type perturbation $T_z^* - \sum_{i=1}^p T_z^* \vp_i \otimes \vp_i$.

 

\section{Almost invariant subspaces}\label{sec: almost IS}

First, we recall the definition of almost invariant subspaces (see \cite{CGP}). Let $T \in \clb(\clh)$, and let $\clm \subseteq \clh$ be a closed subspace. Then $\clm$ is almost invariant for $T$ if there exists a finite-dimensional subspace $\clf\subseteq \clh$ such that
\[
T \clm \subseteq \clm \oplus \clf.
\]
Recall also that the defect of $\clm$ is the lowest possible dimension of such a space $\clf$. This section aims to link almost invariant subspaces of $T_z^*$ with the invariant subspaces of finite-rank perturbations for $T_z^*$. The link is quite natural and could even be considered a simple observation. However, this in particular gives a new insight into the theory of almost invariant subspaces and unifies the perturbation theory from an invariant subspace stance. We present the result in finer generality. We will assume $\clh$ is a Hilbert space, and, as usual, $\{u_i\}_{i=1}^m$ and $\{v_i\}_{i=1}^m$ are orthonormal sets in $\clh$.

\begin{theorem}\label{th-almost}
Let $T \in \clb(\clh)$ and let $\clm \subseteq \clh$ be a closed subspace. If $\clm$ is invariant under $T - \sum_{i=1}^{m} v_i \otimes u_i$, then $\clm$ is almost invariant under $T$ with defect at most $m$. Conversely, if $\clm $ is almost invariant under $T$ with defect $m$, then $\clm$ is invariant under $T - \sum_{i=1}^{m} f_i \otimes T^*f_i$, where $\{f_i\}_{i=1}^m$ is an orthonormal basis for the defect space.
\end{theorem}
\begin{proof}
Suppose $\clm \subseteq \clh$ is invariant under $T - \sum_{i=1}^{m} v_i \otimes u_i$, that is
\[
(T - \sum_{i=1}^{m} v_i \otimes u_i) \clm \subseteq \clm.
\]
Note that
\[
(v_i \otimes u_i)\clm = \{\langle f, u_i\rangle v_i: f \in \clm\} = \C v_i,
\]
for all $i=1, \ldots, m$. Therefore
\[
\begin{split}
T \clm & = \Big((T - \sum_{i=1}^{m} v_i \otimes u_i) + \sum_{i=1}^{m} v_i \otimes u_i\Big) \clm\\
& \subseteq \clm + \text{span}\{v_i: i=1, \ldots, m\},
\end{split}
\]
and hence, $\clm$ is almost invariant under $T$ with defect space $\clf$ such that
\[
\clf \subseteq \text{span}\{v_i: i=1, \ldots, m\}.
\]
This also says, in particular, that the $\dim \clf$ is at most $m$. For the converse direction, assume that $\clm $ is almost invariant under $T$ with defect $m$ and that $\clf$ is the $m$-dimensional defect space having an orthonormal basis, say $\{f_i:1\leq i\leq m\}$. Therefore, we have
\[
T \clm \subseteq \clm \oplus \clf.
\]
Fix $f \in \clm$. Since $Tf \in \clm \oplus \clf$, there exist $f_\clm \in \clm$ and $f_\clf \in \clf$ such that
\[
Tf = f_\clm \oplus f_\clf.
\]
Therefore, we have
\[
\begin{split}
(T - \sum_{i=1}^{m} f_i \otimes T^* f_i) f & = Tf - \sum_{i=1}^{m} \langle f, T^* f_i\rangle f_i
\\
& = Tf - \sum_{i=1}^{m} \langle f_\clm \oplus f_\clf, f_i\rangle f_i
\\
& = Tf - \sum_{i=1}^{m} \langle f_\clf, f_i\rangle f_i.
\end{split}
\]
Also note that
\[
f_\clf = \sum_{i=1}^{m} \langle f_\clf, f_i\rangle f_i,
\]
and consequently
\[
\begin{split}
(T - \sum_{i=1}^{m} f_i \otimes T^* f_i) f & = Tf - \sum_{i=1}^{m} \langle f_\clf, f_i\rangle f_i
\\
& = (f_\clm \oplus f_\clf) - f_\clf
\\
& = f_\clm
\\
& \in \clm.
\end{split}
\]
This proves that $\clm$ is invariant under $T - \sum_{i=1}^{m} f_i \otimes T^* f_i$.
\end{proof}

In particular, we have the following in the setting of $T^*_z$: Let $\clm$ be a closed subspace of $H^2$, and let $\{u_i\}_{i=1}^m$ and $\{v_i\}_{i=1}^m$ be orthonormal sets in $H^2$. If $\clm$ is invariant under $T_z^* - \sum_{i=1}^{m} v_i \otimes u_i$, then $\clm$ is almost invariant under $T_z^*$ with defect at most $m$. Conversely, if $\clm $ is almost invariant under $T_z^*$ with defect $m$, then $\clm$ is invariant under $T_z^* - \sum_{i=1}^{m} f_i \otimes T_z f_i$, where $\{f_i:1\leq i\leq m\}$ is an orthonormal basis of the defect space.

In Theorem \ref{Main-th}, we have already classified the invariant subspaces of finite-rank perturbations of $T_z^*$ of the above type. As a result, the above observation, along with Theorem \ref{Main-th}, easily implies the classification of almost invariant subspaces of $T_z^*$:

\begin{corollary}\label{al-crlre}
Let $\clm \subseteq H^2$ be a closed subspace. If $\clm$ is almost invariant under $T_z^*$ with defect $m$, then there exist a natural number $p' \leq p+1$, an inner function $\Psi \in H^\infty_{\clb(\C^{p'}, \C^{p+1})}$, and a unitary operator $\mathcal{J} : \clk_\Psi \to \clm$ such that
\begin{equation}\label{eqn: IS rep 1A}
\clm = [F_0,1]_{1\times (p+1)} \clk_\Psi,
\end{equation}
where $F_0 = [f_1, \ldots, f_p]$, $\{f_i\}_{i=1}^p$ is an orthonormal basis for $\text{span}\{P_\clm(T_z u_i)\}_{i=1}^m$, and $\{u_i\}_{i=1}^m$ is an orthonormal basis for the defect space, and
\begin{equation}\label{eqn: IS rep 3A}
\mathcal{J}(k_1, \ldots ,k_{p+1})= f_1k_1+ \cdots +f_pk_p +k_{p+1},
\end{equation}
for all $(k_1, \ldots ,k_{p+1}) \in \clk_\Psi$. Moreover
\begin{equation}\label{eqn: IS rep 2A}
\mathcal{K}_\Psi = \{(F, f_0)\in H^2_{\C^p} \oplus H^2 : F_0 F + f_0 \in \clm \textit{ and } \|{F}\|^2 +\|{f_0}\|^2 = \|F_0 F + f_0\|^2\}.
\end{equation}
Conversely, if $\clm$ has the representation \eqref{eqn: IS rep 1A} for some orthonormal set of vectors $\{f_i\}_{i=1}^p$ for $\clm$, an inner function $\Psi \in H^\infty_{\clb(\C^{p^\prime},\C^{p+1})}$ with $\clk_\Psi$ as in \eqref{eqn: IS rep 2A}, and a unitary $\clj$ as in \eqref{eqn: IS rep 3A}, then $\clm$ is almost invariant under $T_z^*$.
\end{corollary}

We remark that Chalendar, Gallardo-Guti\'{e}rrez, and Partington previously classified almost invariant subspaces of $T_z^*$ \cite[Theorem 3.4]{CGP}. However, the present result differs substantially. First, it connects with a topic that appears to have a different flavor, namely finite-rank perturbations of $T_z^*$; second, this comes as a byproduct of finite-rank perturbations of $T_z^*$. Third, the connection with the perturbation theory made the representations of almost invariant subspaces more explicit. For instance, the representation of the model space $\clk_\Psi$ is a new addition compared to the one derived in \cite[Theorem 3.4]{CGP}.

\section{Nearly invariant subspaces}\label{sec3}

The aim of this section is to establish a connection between our previous results on finite-rank perturbations and the concept of nearly invariant subspaces. At first, we set up the notations. Let $\{z_k\}_{k=1}^n$ be $n$-points (repetition is allowed) in $\D$. Assume that $z_j = 0$ for some $j=1, \ldots, n$. Define
\[
B_n(z)=\prod_{k=1}^n\dfrac{z_k -z}{1-\bar{z}_k z} \qquad (z \in \D),
\]
the Blaschke product corresponding to $\{z_k\}_{k=1}^n$. Clearly, $B_n \in H^\infty$ is an inner function, $B_n(0) = 0$, and
\[
\dim \clk_{B_n} = n.
\]
Throughout this section, we will fix the Blaschke product mentioned above. We recall the following from Definition \ref{def: nearly B}: A closed subspace $\clm \subseteq H^2$ is nearly $T^*_{z,B_n}$-invariant if $T_z^* f \in \clm$ for all
\[
f\in \clm \cap B_n H^2.
\]
If $n=1$, then $B_1(z) = z$, and hence, nearly $T^*_{z,B_1}$-invariant subspaces are precisely nearly $T_z^*$-invariant. For the basic preparation, we recall that $H^2$ is a reproducing kernel Hilbert space corresponding to the Szeg\"{o} kernel
\[
k(z,w) = \frac{1}{1-\bar{w}z},
\]
for all $z,w \in \D$. The \textit{kernel function} $k_w$ at $w \in \D$ is defined by
\[
k_w(z) = k(z,w),
\]
for all $z \in \D$. We also know that the set of kernel functions $\{k_w: w \in \D\}$ is total in $H^2$ and satisfy the \textit{reproducing property}
\[
f(w) = \langle f, k_w \rangle,
\]
for all $f \in H^2$ and $w \in \D$. The following general lemma concerns dimension estimations:

\begin{lemma}\label{dim}
Let $\clm$ be a closed subspace of $H^2$. Then
\[
\dim \left(\clm \ominus (\clm \cap B_n H^2)\right) \leq n.
\]
\end{lemma}
\begin{proof}
Recall that $\dim \clk_{B_n} = n$. Suppose $\{f_1, \ldots, f_n\}$ is an orthonormal basis for $\clk_{B_n}$. Set
\[
\cll = span\{P_\clm f_1, \ldots, P_\clm f_n\}.
\]
Clearly, $\cll \subseteq \clm$. Now, for each $f \in \clm \cap B_n H^2$ and $j=1, \ldots, n$, we have
\[
\langle P_\clm f_j, f \rangle = \langle f_j, f \rangle = 0,
\]
as $f_j \perp B_n H^2$. This implies that $\cll \subseteq \clm \ominus (\clm \cap B_n H^2)$. We claim that this inclusion is part of equality. To see this, pick $f \in \clm$ and assume that $\langle f, P_\clm f_j \rangle = 0$ for all $j= 1, \ldots, n$. This is equivalent to say that  $\langle f, f_j \rangle = 0$ for all $j= 1, \ldots, n$, and hence $f \in B_n H^2$. Thus $f \in \clm \cap B_n H^2$, and so $\cll = \clm \ominus (\clm \cap B_n H^2)$, completing the proof of the lemma.
\end{proof}

We are now ready to prove the invariant subspace theorem. The proof follows the same structure as Theorem \ref{Main_th_N}. Hence, the subsequent proof will be concise and will exclude details that resemble the proof of Theorem \ref{Main_th_N}.

\begin{theorem}\label{B-inv-thm}
Let $\clm \subseteq \hdcc $ be a nontrivial closed subspace of $H^2$. Then $\clm$ is nearly $T^*_{z,B_n}$-invariant if and only if there exist natural numbers $r' \leq r\leq n$, an inner function $\Theta \in H^\infty_{\clb(\C^{r'},\C^r)}$ with $\Theta(0)=0$, and a unitary operator $\clj: \clk_\Theta \raro \clm$ such that
\[
\clm = [g_1,\ldots , g_r]_{1\times r} \clk_\Theta,
\]
where $\{g_i\}_{i=1}^r$ is an orthonormal basis for $\clm \ominus (\clm \cap B_n \hdcc)$, and
\[
\clj(k_1, \ldots ,k_r) = g_1k_1 + \cdots + g_r k_r,
\]
for all $(k_1, \ldots ,k_r) \in \clk_\Theta$.
\end{theorem}
\begin{proof}
Suppose $\clm$ is nearly $T^*_{z,B_n}$-invariant. We claim that
\[
\clm \nsubseteq B_n \hdcc.
\]
Indeed, if $\clm \subseteq B_n\hdcc$, then $\clm \subseteq z\hdcc$ as $B_n(0)=0$. So, for any
\[
f= \sum_{n=0}^{\infty}a_nz^n \in \clm,
\]
we have $a_0=0$. By the definition of nearly $T^*_{z,B_n}$-invariant, we have $T_z^* f \in \clm$, which gives $\sum_{t=1}^{\infty}a_t z^{t-1}\in \clm\subseteq z\hdcc$. But now $a_1 =0$, and hence, in a similar way, we conclude that $a_t =0$ for all $t\geq 0$. Then $f=0$ implying that $\clm =\{0\}$. This proves the claim, and consequently
\[
\clm \ominus (\clm \cap B_n \hdcc) \neq \{0\}.
\]
This and Lemma \ref{dim} then implies that
\[
r := \dim  (\clm \ominus (\clm \cap B_n \hdcc)) \in [1, n].
\]
Suppose $\{g_1, \ldots , g_r\}$ be an orthonormal basis for $\clm \ominus (\clm \cap B_n \hdcc)$. We decompose $\clm$ as
\[
\clm = \big(\clm \ominus (\clm \cap B_n \hdcc)\big) \oplus \big(\clm \cap B_n \hdcc \big).
\]
Fix $f \in \clm$, and write
\[
f= f_0 \oplus h_1
\]
for some $f_0 \in \clm \ominus (\clm \cap B_n \hdcc)$ and $h_1 \in \clm \cap B_n \hdcc$. Assume that $\{g_i\}_{i=1}^r$ is an orthonormal basis for $\clm \ominus (\clm \cap B_n \hdcc)$, and set $G_0 = [g_1, \ldots, g_r]_{1\times r}$. Then
\[
f_0 = a_{01}g_1 + \cdots +a_{0r}g_r = G_0A_0,
\]
where $A_0 = [a_{01}, \ldots, a_{0r} ]^t \in \C^r$. We also have that
\[
\|{f}\|^2 = \|{A_0}\|^2 +\|{h_1}\|^2 .
\]
Since $h_1 \in \clm \cap B_n \hdcc$, by the definition of nearly $T^*_{z,B_n}$-invariant subspace, $T_z^* h_1 \in \clm$, where, on the other had, $B_n (0)=0$ implies $h_1 = zf_1$ for some $f_1 \in \clm$. Then
\[
f= G_0A_0 \oplus zf_1,
\]
and $\|{f}\|^2 = \|{A_0}\|^2 +\|{f_1}\|^2$. We are now in exactly the same situation as in the proof of Theorem \ref{Main_th_N}, specifically the identity obtained in \eqref{eqn: f = FA+g}. Thus, we take the argument further and conclude similarly that
\[
f = G_0 K,
\]
with $\|{f}\|^2 = \|{K}\|^2$, where, $K = \sum_{t=0}^{\infty}A_t z^t \in H^2_{\C^r}$. With this in mind, define
\[
\clk: = \left\{K \in H^2_{\C^r}: G_0K \in \clm, \text{ and } \|G_0K\| = \|{K}\| \right\} \subseteq H^2_{\C^r}.
\]
Pick $K \in \clk$. Then there exists $f\in \clm$ such that $f = G_0K$. By the above construction, we know $f = G_0 A_0 \oplus zf_1$ with $A_0 = K(0)$. Then
\[
G_0\left(\frac{K(z)-K(0)}{z}\right) = f_1 \in \clm,
\]
implying $T_z^*K \in \clk$. Therefore by the Beurling-Lax theorem, there exist a natural number $r' \leq r$ and an inner function $\Theta \in H^\infty_{\clb(\C^{r^\prime},\C^r)}$ such that $\clk = \clk_\Theta$. In other words, $\clm = G_0 \clk_\Theta$. Our next aim is to show $\Theta(0) = 0$. Denote by $\{e_i\}_{i=1}^r$ the standard orthonormal basis for $\C^r$. For any $i=1, \ldots, r$, if we set $K = 1 \otimes e_i$, then $G_0 K = g_i$ together with $\|g_i\| = \|G_0K\| = \|e_i\| = 1$ implies $1 \otimes e_i \in \clk_\Theta$. This means
\[
1 \otimes \eta \in \clk_\Theta,
\]
for all $\eta \in \C^r$, and hence
\[
\Theta(0)^* \eta =  T_\Theta^*(1 \otimes \eta) = 0,
\]
for all $\eta \in \C^r$. Therefore, $\Theta(0) = 0$. The existence of unitary $\clj$ follows from the description of the space $\clk$.

\noindent For the converse direction, pick $f = G_0 K \in \clm$ for some $ K=(k_1, \ldots ,k_r) \in \clk_\Theta$. Assume that $f\in \clm \cap B_n \hdcc$. Since $\{g_i\}_{i=1}^r$ is an orthonormal basis for $\clm \ominus (\clm \cap B_n \hdcc)$, for $i \in \{1, \ldots , r\}$, we have
\[
\langle f, g_i\rangle = \langle g_1k_1 + \cdots + g_r k_r, g_i \rangle = 0,
\]
which implies $k_i(0) = 0$, and hence $K(0) = 0$. Now $f = G_0 K$ implies that $f(0) = 0$, and then from $T_z^* f = \frac{f - f(0)}{z}$, it follows that
\[
T_z^* f = \frac{G_0 K}{z} = G_0 \frac{K -K(0)}{z} = G_0 ((T_z^* \otimes I_{\C^r}) K) \in G_0 \clk_\Theta = \clm,
\]
from which we conclude that $\clm$ is nearly $T^*_{z,B_n}$-invariant.
\end{proof}

Considering the case $B_1(z)= z$, we can recover the structure of classical nearly $T_z^*$-invariant subspaces, which has been described by Hitt \cite{Hitt} and further modified by Sarason \cite{Sarason}. In this case, nearly $T_z^*$-invariant subspaces are of the form $gK_\vp$ where $\clk_\vp \subseteq H^2$ is a model space for some inner function $\vp \in H^\infty$ and $g$ is an isometric multiplier on $\clk_\vp$.

Lastly, we establish a connection between invariant subspaces of finite-rank perturbations of $T_z^*$ and nearly $T^*_{z,B_n}$-invariant subspaces. This is comparable to the relationship between invariant subspaces of finite-rank perturbations and almost invariant subspaces of $T_z^*$. The preceding theorem is the underlying basis for the result that follows.

\begin{corollary}\label{cor: pert and nearly}
Let $\clm$ be a closed subspace of $H^2$. If $\clm$ is invariant under $T_z^* - \sum_{i=1}^{n} T_z^* k_{z_i} \otimes k_{z_i}$, then $\clm$ is nearly $T^*_{z,B_n}$-invariant. Conversely, if $\clm$ is nearly $T^*_{z,B_n}$-invariant, then $\clm$ is invariant under $T_z^* - \sum_{i=1}^r T_z^* g_i \otimes g_i$, where $\{g_i\}_{i=1}^r$ is an orthonormal basis for $\clm \ominus (\clm \cap B_n \hdcc)$.
\end{corollary}
\begin{proof}
Suppose $\clm$ is invariant under $T_z^* - \sum_{i=1}^{n} T_z^* k_{z_i} \otimes k_{z_i}$. Let $f \in \clm \cap B_n H^2$. There exists $g \in H^2$ such that $f = B_ng \in \clm$. Then
\[
(T_z^* k_{z_i} \otimes k_{z_i}) f = \langle f, k_{z_i}\rangle T_z^* k_{z_i}  = f(z_i) T_z^* k_{z_i} = B_n(z_i) g(z_i) T_z^* k_{z_i} = 0,
\]
as $B_n(z_i) = 0$ for all $i=1, \ldots, n$. By assumption, we have $(T_z^* - \sum_{i=1}^{n} T_z^* k_{z_i} \otimes k_{z_i})f \in \clm$, and hence
\[
T_z^*f = (T_z^* - \sum_{i=1}^{n} T_z^* k_{z_i} \otimes k_{z_i}) f + (\sum_{i=1}^{n} T_z^* k_{z_i} \otimes k_{z_i})f =  (T_z^* - \sum_{i=1}^{n} T_z^* k_{z_i} \otimes k_{z_i}) f \in \clm,
\]
completes the proof of the fact that $\clm$ is nearly $T^*_{z,B_n}$-invariant. For the converse direction, assume that $\clm$ is nearly $T^*_{z,B_n}$-invariant. By Theorem \ref{B-inv-thm}, there exist natural numbers $r' \leq r\leq n$, an inner function $\Theta \in H^\infty_{\clb(\C^{r'},\C^r)}$ with $\Theta(0)=0$, and a unitary operator $\clj: \clk_\Theta \raro \clm$ such that
\[
\clm = [g_1,\ldots , g_r]_{1\times r} \clk_\Theta,
\]
where $\{g_i\}_{i=1}^r$ is an orthonormal basis for $\clm \ominus (\clm \cap B_n \hdcc)$, and
\[
\clj(k_1, \ldots ,k_r) = g_1k_1 + \cdots + g_r k_r,
\]
for all $(k_1, \ldots ,k_r) \in \clk_\Theta$. The remaining part of the proof follows the same line as the converse part of Theorem \ref{Main-th}.
\end{proof}

In a sense, the above result also yields examples of nearly $T^*_{z,B_n}$-invariant subspaces. We shall see more concrete examples in the following section. 

\section{Perturbed Toeplitz operators}\label{sec4}

The purpose of this section is to provide some examples of nearly $T^*_{z,B_n}$-invariant subspaces. From this perspective, we remind the reader that the kernels of Toeplitz operators are nearly $T_z^*$-invariant subspaces \cite{Sarason 1}. Along this line, we prove that the kernels of certain finite-rank perturbations of Toeplitz operators are nearly $T^*_{z,B_n}$-invariant. Consequently, in addition to the theory of finite-rank perturbations of $T_z^*$, the results of this section can be considered an additional justification for the new notion of nearly $T^*_{z,B_n}$-invariant subspaces.
 

We adopt the setting of Section \ref{sec3} and continue our discussion with nearly $T^*_{z,B_n}$-invariant subspaces. Therefore, $B_n$ is a fixed Blaschke product with $\{z_k\}_{k=1}^n$ as the zero set, and we assume that $B_n(0) = 0$. Recall that
\[
\dim \clk_{B_n} = n,
\]
where $\clk_{B_n} = H^2/B_n H^2$ is the model space. Given an arbitrary orthonormal basis $\{f_i\}_{i=1}^n$ of the model space $\clk_{B_n}$, we will be interested in the finite-rank perturbation of
\[
T_\vp + \sum_{i=1}^{n} f_i \otimes T_z^* f_i.
\]
We now prove that the kernel of this is nearly $T^*_{z,B_n}$-invariant subspace:

\begin{theorem}\label{thm: finite-rank pert TO}
Let $\vp \in L^\infty$. Then, for each orthonormal basis $\{f_i\}_{i=1}^n$ for the model space $\clk_{B_n}$, the kernel space
\[
\ker\Big(T_\vp + \sum_{i=1}^{n} f_i \otimes T_z^* f_i\Big),
\]
is nearly $T^*_{z,B_n}$-invariant.
\end{theorem}
\begin{proof}
Pick an arbitrary $f \in \ker(T_\vp + \sum_{i=1}^{n} f_i \otimes T_z^* f_i) \cap B_nH^2$. Our goal is to prove that $T_z^* f \in \ker(T_\vp + \sum_{i=1}^{n} f_i \otimes T_z^* f_i)$. Since
\[
(T_\vp + \sum_{i=1}^{n} f_i \otimes T_z^* f_i)f = 0,
\]
it follows that
\[
T_\vp f + \sum_{i=1}^{n} \langle T_z f, f_i \rangle f_i = 0.
\]
Since $f \in B_nH^2$, we have $T_z f \in B_nH^2$. Also, $\{f_i\}_{i=1}^n$ is an orthonormal basis for $\clk_{B_n}$. This implies
\[
\sum_{i=1}^{n} \langle T_z f, f_i \rangle f_i = 0,
\]
and then $T_\vp f = 0$. Next, we recall that $T_z T_z^* = I - P_{\C}$, where $P_{\C}$ denotes the orthogonal projection of $H^2$ onto the one-dimensional space of all constant functions. In the present scenario, $f$ is in $B_nH^2$ and $B_n(0) = 0$, which implies $f(0) = 0$. Therefore, $P_{\C} f = 0$ and hence
\[
T_z T_z^* f = f.
\]
We use this and the Toeplitz operator identity $T_z^*T_\vp T_z = T_\vp $ (see \eqref{B-H}) to compute
\begin{align*}
(T_\vp + \sum_{i=1}^{n} f_i \otimes T_z^* f_i) T_z^*f & = T_\vp (T_z^* f) + \sum_{i=1}^{n}\langle T_z^*f, T_z^* f_i \rangle f_i
\\
& = T_z^*T_\vp T_z (T_z^*f) + \sum_{i=1}^{n}\langle T_z T_z^*f , f_i \rangle f_i
\\
& = T_z^* T_\vp f + \sum_{i=1}^{n}\langle f , f_i \rangle f_i.
\end{align*}
Since $f \in B_n H^2$ and $\{f_i\}_{i=1}^n$ is an orthonormal basis for $\clk_{B_n}$, it is clear that
\[
\sum_{i=1}^{n}\langle f , f_i \rangle f_i = 0.
\]
Then $T_\vp f = 0$ implies that
\[
T_z^*f \in \ker (T_\vp + \sum_{i=1}^{n} f_i \otimes T_z^* f_i),
\]
and proves the claim that $\ker(T_\vp + \sum_{i=1}^{n} f_i \otimes T_z^* f_i)$ is nearly $T^*_{z,B_n}$-invariant.
\end{proof}

Evidently, the case $B_1(z)= z$ yields the known fact that the kernel of a Toeplitz operator is nearly $T_z^*$-invariant. The following observation is a comparison between nearly $T_z^*$-invariant subspaces and nearly $T^*_{z,B_n}$-invariant subspaces:

\begin{proposition}
Every nearly $T_z^*$-invariant subspace is a nearly $T^*_{z,B_n}$-invariant subspace.
\end{proposition}
\begin{proof}
Let $\clm$ be a nearly $T_z^*$-invariant subspace of $H^2$. Since $B_n(0) = 0$, it follows that $B_nH^2 \subseteq z H^2$, and hence
\[
\clm \cap B_n H^2 \subseteq \clm \cap zH^2.
\]
Therefore, if $f\in \clm \cap B_nH^2$, then $f\in \clm \cap zH^2$, and consequently $T_z^* f \in \clm$. This proves that $\clm$ is $T^*_{z,B_n}$-invariant.
\end{proof}

The converse of the above observation is, of course, not true. For instance, consider $B_n$ as $B_n = z B_{n-1}$, where $B_{n-1}(0)=0$. We consider the $3$-dimensional subspace $\clm$ of $H^2$, where
\[
\clm = span\{B_{n-1}, B_n, zB_n\}.
\]
It is simple to see that $\clm$ is nearly $T^*_{z,B_n}$-invariant but neither $T_z^*$-invariant nor nearly $T_z^*$-invariant.

In terms of more examples of nearly $T^*_{z,B_n}$-invariant subspaces, we note that a class of Schmidt subspaces of Hankel operators are nearly $T_z^*$-invariant subspaces \cite{GaPu21, GaPu20}. Given the above illustrations, one would naturally predict that the same would apply to finite-rank perturbations of Hankel operators or perturbations of operators that are related to Hankel operators. It would certainly be fascinating to identify such class operators.

In closing, we remark that this paper effectively connects invariant subspaces of three types of operators: finite-rank perturbations of $T_z^*$, the Blaschke-based backward shift $T_{z, B_n}^*$, and almost invariant subspaces of $T_z^*$. We have enhanced the nearly invariant subspace techniques of Hayashi \cite{Hayashi}, Hitt \cite{Hitt}, and Sarason \cite{Sarason}. This enhancement allowed us to tackle the fundamentally challenging invariant subspace problem for perturbed backward shift operators. One should naturally expect further developments along these lines.

\vspace{0.1in}

\noindent\textbf{Acknowledgement:}
The research of the first named author is supported by the Theoretical Statistics and Mathematics Unit, Indian Statistical Institute, Bangalore, India. The research of the second named author is supported in part by TARE (TAR/2022/000063) by SERB, Department of Science \& Technology (DST), Government of India.


\begin{thebibliography}{99}



\bibitem{Beurling}
A. Beurling, {\em On two problems concerning linear transformations in Hilbert space}, Acta Math. 81 (1949), 239-255.

\bibitem{BeTi00}
C. Benhida and D. Timotin, {\em Finite rank perturbations of contractions}, Integral Equ. Oper. Theory. 36 (2000), 253-268.

\bibitem{Partington}
C. C\^amara and J. Partington, {\em Near invariance and kernels of Toeplitz operators}, J. Anal. Math. 124 (2014), 235–260.

\bibitem{CGP}
I. Chalendar, E. Gallardo-Guti\'{e}rrez and J. Partington,  {\em A Beurling Theorem for almost-invariant subspaces of the shift operator}, J. Operator Theory. 83 (2020), 321-331.

\bibitem{Fricain}
E. Fricain, A. Hartmann and W. Ross, {\em Range spaces of co-analytic Toeplitz operators}, Canad. J. Math. 70 (2018), 1261–1283.

\bibitem{Gallardo}
E. Gallardo-Guti\'{e}rrez and J. Gonz\'{a}lez-Do\~{n}a, {\em Finite rank perturbations of normal operators: spectral idempotents and decomposability}, J. Funct. Anal. 285 (2023), Paper No. 110148, 50 pp.

\bibitem{GaPu21}
P. G\'{e}rard and A. Pushnitski, {\em The structure of Schmidt subspaces of Hankel operators: a short proof}, Studia Math. 256 (2021), 61–71.


\bibitem{GaPu20}
P. G\'{e}rard and A. Pushnitski, {\em Weighted model spaces and Schmidt subspaces of Hankel operators}, J. Lond. Math. Soc. (2) 101 (2020), 271--298.


\bibitem{HS}
A. Hartmann and K. Seip, {\em Extremal functions as divisors for kernels of Toeplitz operators}, J. Funct. Anal. 202 (2003), 342–362.


\bibitem{Hayashi}
E. Hayashi, {\em The kernel of a Toeplitz operator}, Integral Equ. Oper. Theory. 9 (1986), 588--591.

\bibitem{Hitt}
D. Hitt, {\em Invariant subspaces of $H^2$ of an annulus}, Pacific J. Math. 134 (1988), 101--120.

\bibitem{Kato}
T. Kato, {\em Perturbation theory of linear operators}, 2nd edn. Springer, New York, 1976.

\bibitem{LP}
Y. Liang and J. Partington, {\em Representing kernels of perturbations of Toeplitz operators by backward shift-invariant subspaces}, Integral Equ. Oper. Theory. 92 (2020), Paper No. 35, 16 pp.

\bibitem{NaFo70}
B. Sz.-Nagy and C. Foia\c{s}, {\em Harmonic analysis of operators on Hilbert space}, North Holland, Amsterdam, 1970.


\bibitem{Nord1}
E. Nordgren, P. Rosenthal and F. Wintrobe, {\em Invertible composition operators on $H^p$}, J. Funct. Anal. 73 (1987), 324–344.

\bibitem{Nikolski}
N. Nikolski, {\em Operators, functions, and systems: an easy reading, Vol. 1 \& 2}, Mathematical Surveys and Monographs, American Mathematical Society, Providence, RI, 2002.

\bibitem{Sarason 1}
D. Sarason, {\em Kernels of Toeplitz operators}, Oper. Theory Adv. Appl. 71 (1994), 153–164.

\bibitem{Sarason}
D. Sarason, {\em Nearly invariant subspaces of the backward shift}, Oper. Theory Adv. Appl. 35 (1988), 481--493.
\end{thebibliography}
\end{document}